%% file: main.tex
\documentclass[12pt, reqno]{amsart}
\input{def}

\usepackage{xcolor}

\begin{document}

\author[T.~Moseeva]{Tatiana Moseeva}
\address{Tatiana Moseeva, Leonhard Euler International Mathematical Institute, Russia}
\email{polezina@yandex.ru}

\author[A.~Tarasov]{Alexander Tarasov}
\address{Alexander Tarasov, Saint Petersburg State University, Russia}
\email{science.tarasov@gmail.com}

\author[D.~Zaporozhets]{Dmitry Zaporozhets}
\address{Dmitry Zaporozhets, St.~Petersburg Department of Steklov Institute of Mathematics, Russia}
\email{zap1979@gmail.com}

\title[Random sections of spherical convex bodies]{Random sections of \\ spherical convex bodies}

\begin{abstract}
Let $K\subset\mathbb S^{d-1}$ be a convex spherical body. Denote by $\Delta(K)$ the distance between two random points in $K$ and denote by $\sigma(K)$ the length of a random chord of $K$. We explicitly express the distribution of $\Delta(K)$ via the distribution of $\sigma(K)$.  From this  we find the density of distribution of $\Delta(K)$ when $K$ is a spherical cap. 
\end{abstract}

\maketitle

\section{Introduction}
For $d\in\mathbb N$ fix some $k\in\{1,\dots,d-1\}$ and denote by $G_{d,k}$ (respectively, $A_{d,k})$  the set of all linear (respectively, affine) $k$-planes  in $\R^d$ equipped with the unique Haar measure invariant with respect to rotations (respectively, rigid motions) normalized by
\begin{align*}
    \nu_{d,k}\left(\left\{L\in G_{d,k}\right\}\right)=1,
\end{align*}
respectively,
\begin{align*}
    \mu_{d,k}\left(\left\{E\in A_{d,k}:\,E\cap \B^d\ne\emptyset\right\}\right)=\kappa_{d-k},
\end{align*}
where is  the $k$-dimensional unit ball and $\kappa_k:=|\mathbb B^k|$. By $|\cdot|$ we denote the volume of the appropriate dimension, by which we understand the Lebesgue measure with respect to the affine hull of a set. 

When $k=1$, we deal with the set of lines. The seminal Crofton formula says that for any convex body (convex compact set with non-empty interior) $K$ we have
\begin{align*}
    \int_{A_{d,1}} |K\cap E|^{d+1}\,\nu_{d,1}(\mathrm{d}E)=\frac{d(d+1)}{2d\kappa_d}|K|^2.
\end{align*}
It has been obtained by Crofton~\cite{mC85} for $d=2$ and later extended by Hadwiger~\cite{hH52} to all $d$.
What's less well-known is the following generalization which  has been  independently obtained in~\cite[Eq.~(21)]{gC67} and~\cite[Eq.~(34)]{jK69}: for $p>-d$,
\begin{align*}
    \int\limits_{A_{d,1}}|K\cap E|^{d+p+1}\,\mu_{d,1}(\mathrm{d}E)=\frac{(d+p)(d+p+1)}{2d\kappa_d}\int\limits_{K^2}{|\bx_0-\bx_1|^p\,\dd\bx_0\dd\bx_1}.
\end{align*}
In probabilistic language, it establishes a connection between the moments of two random variables:
\begin{align*}
    \E \sigma^{d+p+1}=\frac{(d+p)(d+p+1)}{2\kappa_{d-1}}\cdot\frac{|K|^2}{|\partial K|}\E\Delta^p,
\end{align*}
where $\sigma=\sigma(K)$ denotes the length of the intersection of $K$ with the random line uniformly distributed among all lines from $A_{d,1}$ intersected $K$ and $\Delta=\Delta(K)$ denotes the distance between two random points independently and uniformly chosen in $K$. By $|\partial K|$ we denote the surface area of the boundary of $K$ (the $(d-1)$-dimensional Lebesgue measure).

Since a bounded random variable is defined by its moments, we conclude that the distribution of $\Delta$ is defined by the distribution of $\sigma$ (which was not obvious a priori). The explicit connection was derived in~\cite{AO16} for $d=2$ and in~\cite{tM19} for any $d$:
\begin{align}\label{1109}
    f_\Delta(t)=\frac{t^{d-1}}{|K|}\bigg(d\kappa_d-\kappa_{d-1}\frac{|\partial K|}{|K|}\int_0^t(1-F_\sigma(s))\dd s\bigg),
\end{align}
where $f_\Delta$ is the density function of the distribution of $\Delta$ and $F_\sigma$ is the distribution function of $\sigma$. To obtain this result, the authors used the polar coordinates for $d=2$ and the affine Blaschke--Petkantchin formula~\cite[Theorem~7.2.7]{SW08} in general case. 

\medskip

The goal of this paper to obtain a \emph{spherical} analogue of~\eqref{1109}. In the next section we first introduce some basic notation and facts form the spherical integral geometry and then formulate our main result.

\section{Main result}
Now we turn to spherical geometry. Since we are not going to return to the Euclidean case anymore, we will keep some notation for spherical counterparts.

Denote by $\Sd$ the $(d-1)$-dimensional unit sphere and let $\omega_d$ denote its $(d-1)$-dimensional Lebesgue measure: $\omega_d:=|\Sd|=d\kappa_d$.

Let $K\subset\Sd$ be  a spherical convex body, which means that it can be represented as $K=\Sd\cap C$, where $C$ is a line-free closed convex cone in $\R^d$.

Denote by $\Delta=\Delta(K)$ the spherical distance between two random points independently and uniformly chosen in $K$. Formally, $\Delta$ is defined as an angle between two lines independently and uniformly distributed among the lines from $G_{d,1}$ which intersect $K$.

Also define $\sigma=\sigma(K)$ to be the spherical length (1-dimensional Lebesgue measure) of the intersection of $K$ with the $2$-plane  uniformly distributed among $2$-planes form $G_{d,2}$ which intersect $K$.

Our main result is the following spherical version of~\eqref{1109}.

\begin{theorem}\label{1430}
For any spherical body $K\subset\Sd$, the density function of distribution of $\Delta(K)$ can be expressed in terms of the distribution function of $\sigma(K)$ as follows:
\begin{align*}
     f_\Delta(t)=\frac{(\sin t)^{d-2}}{|K|}\bigg(w_{d-1}-\frac{w_d}{2\pi}\kappa_{d-1}\frac{|\partial K|}{|K|}\int_0^t(1-F_\sigma(s))\dd s\bigg).
\end{align*}
\end{theorem}
The proof is given in the next section.
As an application let us find the density of the distribution between two random points in a spherical cap.
\begin{corollary}\label{1944}
    Let $K$ be a spherical cap of spherical radius $r<\frac{\pi}{2}$. Then
\begin{align*}
     f_\Delta(t)=\omega_{d-1}\frac{(\sin t)^{d-2}}{|K|}\bigg(1-\frac{w_d}{2\pi}\kappa_{d-1}\frac{1}{|K|}\int_0^t \left(1 - \frac{(\cos r) ^2}{(\cos \frac{s}{2}) ^2}\right)^{\frac{d-2}{2}} \dd s\bigg).
\end{align*}
\end{corollary}
\begin{proof}
It is straightforward to check that the spherical length of $K \cap L)$ is less than $s$  if an only if $L \cap K_s = \emptyset$, where $K_s$ is a spherical cap with the same center as $K$ and with spherical radius $\arccos\left({\cos r}/{\cos \frac{s}{2}}\right)$. 
Thus in view of~\eqref{measure of planes} from below, 
\begin{align*}
   1- F_{\sigma}(s)= \frac{\mu_{d,2}\{L \cap K_s \neq \emptyset\}}{\mu_{d,2}\{L \cap K \neq \emptyset\}}=\frac{|\partial K_s|}{|\partial K|},
\end{align*}
and applying Theorem~\ref{1430} together with
\begin{align*}
    {|\partial K_s|}={\omega_{d-1}} \left(\sin\arccos\left(\frac{\cos r}{\cos \frac{s}{2}}\right)\right)^{d-2} = \omega_{d-1}\left(1 - \frac{(\cos \alpha) ^2}{(\cos \frac{t}{2}) ^2}\right)^{\frac{d-2}{2}}
\end{align*}
concludes the proof.
\end{proof}
If $d$ is even, then $f_\Delta$ for a spherical cap can be expressed in terms of the  elementary trigonometric functions.
\begin{corollary}
    If $d=2m+2$, then under assumptions of Corollary~\ref{1944} we have
\begin{align*}
     f_\Delta&(t)=\omega_{d-1}\frac{(\sin t)^{d-2}}{|K|}-\frac{\omega_d\omega_{d-1}\kappa_{d-1}}{\pi}\frac{(\sin t)^{d-2}}{|K|^2}\tan \frac{t}{2}
     \\
     &\times\sum_{k = 0}^{m}(-1)^{m-k}\binom{m}{k}\frac{(2k - 2)!!}{(2k - 1)!!}(\cos r)^{2(m-k)}\left(1 + \sum_{l = 1}^{m - k - 1}\frac{(2l - 1)!!}{(2l)!!}\frac{1}{\left(\cos\frac{t}{2}\right)^{2l}}\right).
\end{align*}
\end{corollary}
\begin{proof}
By the binomial theorem,
\begin{align*}
    \int\limits_0^t \left(1 - \frac{(\cos r) ^2}{(\cos \frac{s}{2}) ^2}\right)^{\frac{d-2}{2}} \dd s = \sum_{k = 0}^{m}(-1)^{m-k}\binom{m}{k}(\cos r)^{2(m-k)}\int\limits_0^t\frac{1}{(\cos \frac{s}{2})^{2(m-k)}} \dd s.
\end{align*}
Denote by $I_{2k}$ the indefinite integral of ${(\cos t)^{-2k}}$.
Applying the reduction formula we get 
\begin{equation*}
    (2k-1)I_{2k}(t) = \tan t\frac{1}{(\cos t)^{2k - 2}} + (2k - 2)I_{2k - 2}(t).
\end{equation*}
Using the fact that $I_2(t) = \tan t$,  we can derive by induction that
\begin{align*}
    I_{2k}(t) = \frac{(2k - 2)!!}{(2k - 1)!!}\tan t\left(1 + \sum_{l = 1}^{k-1}\frac{(2l - 1)!!}{(2l)!!}\frac{1}{(\cos t)^{2l}}\right).
\end{align*}

Therefore, 
\begin{align*}
    \int\limits_0^t &\left(1 - \frac{(\cos r) ^2}{(\cos \frac{s}{2}) ^2}\right)^{\frac{d-2}{2}} \dd s \\
    &= 2 \tan \frac{t}{2}\sum_{k = 0}^{m}(-1)^{m-k}\binom{m}{k}(\cos r)^{2(m-k)}\frac{(2k - 2)!!}{(2k - 1)!!}\left(1 + \sum_{l = 1}^{m - k - 1}\frac{(2l - 1)!!}{(2l)!!}\frac{1}{(\cos (\frac{t}{2})^{2l}}\right).
\end{align*}

\end{proof}

\section{Proof of Theorem~\ref{1430}}
The main ingredient of the proof is the following spherical Blaschke--Petkantchin formula: for any non-negative Borel function $f:(\mathbb{S}^{d-1})^k\to \mathbb{R}$ we have
\begin{align}
\label{blashke-petkanchin}
       \int\limits_{\left( \mathbb{S}^{d-1}\right)^k} &f(x_1,\dots,x_k)\lambda(\mathrm d x_1)\dots \lambda(\mathrm d x_k)
       \\\notag
       &= (k!)^{d-k}b_{d,k}\int\limits_{G_{d,k}}
       \int\limits_{(E \cap \mathbb{S}^{d-1})^{k}} f(x_1,\dots,x_k) |\conv(0,x_1,\dots,x_k)|^{d-k}
       \\\notag
       &\hphantom{k!b_{d,k}\int\limits_{G_{d,k}}
       \int\limits_{(E \cap \mathbb{S}^{d-1})^{k}}}\times\lambda_{L}( \mathrm d x_1)\dots \lambda_{E}(\mathrm d x_k)\mu_{d,k}(\mathrm d L),
\end{align}
where $\lambda,\lambda_E$ are the spherical Lebesgue measures on $\Sd, \Sd\cap L$ of dimensions $d-1, k-1$ respectively, $|\conv(0,x_1,\dots,x_k)|$ denotes the \emph{Euclidean} volume of the convex hull of $0,x_1,\dots,x_k$, and
\begin{equation*}
    b_{d,k} := \frac{\omega_{d-k+1}\cdots \omega_{d}}{\omega_{1}\cdots \omega_{k}}.
\end{equation*}
This formula is a special case of a more general result from~\cite{AZ91}.

Denote by $F_\Delta$ the distribution function of $\Delta(K)$.
By definition,
\begin{align*}
    F_{\sigma}(t) = \frac{\int\limits_{G_{d,2}}\mathbbm{1}[L \cap K \neq \{0\}]\mathbbm{1}[\alpha(K\cap L) < t] \mu_{d,2}(\mathrm dL)}{\mu_{d,2}\{L \in G_{d,2} \mid L\cap K  \neq \emptyset\}},
\end{align*}
\begin{align*}
    1 - F_\Delta(t) = \frac{1}{|K|^2}\int\limits_{\left( S^{d-1}\right)^2}\mathbbm{1}[x_1, x_2 \in K]\mathbbm{1}[\alpha(x_1, x_2) \geqslant t] \lambda(\mathrm d x_1)\lambda(\mathrm d x_2),
\end{align*}
where by $\alpha(x_1, x_2)$ and $\alpha(K\cap L)$ we denote the spherical distance between points $x_1,x_2$ and the spherical length of the chord $K\cap L$.

First let us evaluate $F_\Delta$. Using \eqref{blashke-petkanchin} leads to
\begin{align*}
    &\int\limits_{\left( S^{d-1}\right)^2}\mathbbm{1}[x_1, x_2 \in K]\mathbbm{1}[\alpha(x_1, x_2) \geqslant t] \dd x_1\dd x_2  \\
    & = 2^{d-2}b_{d,2}\int\limits_{G_{d,2}}\int\limits_{\left(S^{d-1}\cap L\right)^2} \mathbbm{1}[x_1, x_2 \in K, \alpha(x_1, x_2) \geqslant t] |\mathrm{conv}(0, x_1, x_2)|^{d-2} 
    \\
    &\hphantom{ 2^{d-2}b_{d,2}\int\limits_{G_{d,2}}\int\limits_{\left(L\cap S^{d-1}\right)^2} }\times\lambda_{L}(\mathrm{d}x_1)\lambda_{L}(\mathrm{d}x_2)\mu_{d,2}(\mathrm{d}L) \\
    & =2^{d-2}b_{d,2}\int\limits_{\alpha(K\cap L) \geqslant t}\int\limits_{\left(K \cap L\right)^2}\mathbbm{1}[\alpha(x_1, x_2 \geqslant t)]|\mathrm{conv}(0, x_1, x_2)|^{d-2}
    \\
    &\hphantom{=2^{d-2}b_{d,2}\int\limits_{\alpha(K\cap L) \geqslant x}\int\limits_{\left(K \cap L\right)^2}}\times\lambda_{L}(\mathrm{d}x_1)\lambda_{L}(\mathrm{d}x_2)\mu_{d,2}(\mathrm{d}L).
\end{align*}
Using the fact that $|\mathrm{conv}(0, x_1, x_2)| = \frac{1}{2}\sin(\alpha(x_1, x_2))$, we get
\begin{align*}
    &\int\limits_{\left(K \cap L\right)^2}\mathbbm{1}[\alpha(x_1, x_2 \geqslant t)]|\mathrm{conv}(0, x_1, x_2)|^{d-2} \lambda_{L}(\mathrm{d}x_1)\lambda_{L}(\mathrm{d}x_2)  \\
    &=\int\limits_{0}^{\alpha(K \cap L)}\int\limits_{0}^{\alpha(K \cap L)}\mathbbm{1}[|\phi_1 - \phi_2| \geqslant t]\left(\frac{1}{2}\sin(|\phi_1 - \phi_2|)\right)^{d-2} \dd \phi_1\dd \phi_2 \\
    &= \frac{1}{2^{d-3}}\int\limits_t^{\alpha(K \cap L)}\int\limits_0^{\phi_1 - t}\sin(\phi_1 - \phi_2)^{d-2} \dd \phi_2\dd \phi_1 = \frac{1}{2^{d-3}}\int\limits_t^{\alpha(K \cap L)}\int\limits_t^{\phi_1}\sin(\phi_2)^{d-2} \dd \phi_2\dd \phi_1.
\end{align*}

Hence, 
$$
    1 - F_{\Delta}(t) = \frac{2b_{d,2}}{|K|^2}\int\limits_{\alpha(K \cap L) \geqslant t}\int\limits_t^{\alpha(E)}\int\limits_t^{\phi_1}\sin(\phi_2)^{d-2} \dd \phi_2 \dd \phi_1.
$$
Let us evaluate the inner double integral. For an integer $n\geq 1$ we have
\begin{align*}
    \sin^n t = \left(\frac{e^{it} - e^{-it}}{2i}\right)^n &= \sum_{k = 0}^n \binom{n}{k}\left(\frac{e^{it}}{2i}\right)^k\left(\frac{-e^{-it}}{2i}\right)^{n-k}
    \\
    &= \frac{1}{(2i)^n}\sum_{k = 0}^n (-1)^{n-k}\binom{n}{k}e^{it(2k - n)}.
\end{align*}
From that we derive an indefinite integral for $ \sin^{d-2}t$,
\begin{align*}
    F(t) :=
    \begin{cases}
         \frac{(-1)^m}{2^{2m-1}}\left[(-1)^m\binom{2m}{m}t + \sum_{k = 0}^{m-1}(-1)^k\binom{2m}{k}\frac{1}{2m - 2k}\sin((2m - 2k)t)\right], \;\; d-2 = 2m, \\
         \frac{(-1)^{m+1}}{2^{2m}}\left[\sum_{k = 0}^{m}(-1)^k\binom{2m+1}{k}\frac{1}{2m + 1 - 2k}\cos((2m + 1 - 2k)t)\right], \;\; d-2 = 2m + 1,
    \end{cases}
\end{align*}
and then an indefinite integral for $F(t)$,
\begin{align*}
    G(t) =
    \begin{cases}
    \frac{(-1)^{m+1}}{2^{2m-1}}\left[(-1)^m\binom{2m}{m}\frac{t^2}{2} + \sum_{k = 0}^{m-1}(-1)^{k}\binom{2m}{k}\frac{1}{(2m - 2k)^2}\cos((2m - 2k)t)\right], \;\; d-2 = 2m, \\
    \frac{(-1)^{m+1}}{2^{2m}}\left[\sum_{k = 0}^{m}(-1)^k\binom{2m+1}{k}\frac{1}{(2m + 1 - 2k)^2}\sin((2m + 1 - 2k)t)\right], \;\; d-2 = 2m + 1.
    \end{cases}
\end{align*}
It follows that
\begin{align*}
    \int\limits_t^{\alpha(K \cap L)}\int\limits_t^{\phi_1}\sin(\phi_2)^{d-2} \dd \phi_2\dd \phi_1 &= \int\limits_t^{\alpha(K \cap L)} F(\phi_1) \dd \phi_1  - F(t)(\alpha(K \cap L) - t) \\
    &= G(\alpha(K \cap L)) - F(t)\alpha(K \cap L) - G(t) + tF(t).
\end{align*}
and finally, we have
\begin{align}\notag
    1 - &F_{\Delta}(t) = \frac{2b_{d,2}}{|K|^2} \left[\int\limits_{\alpha(K \cap L)\geqslant t}\left(G(\alpha(K \cap L))\right) \mu_{d,2}(\mathrm{d}L) -F(t)\int\limits_{\alpha(K \cap L) \geqslant t}\alpha(K \cap L)\mu_{d,2}(\mathrm{d}L) \right. \\\label{whole}
     &+\left. \left( tF(t) - G(t) \right)\int\limits_{\alpha(K \cap L) \geqslant t}\mu_{d,2}(\mathrm{d}L)\right] =:\frac{2b_{d,2}}{|K|^2}\left[ I_1(t) -  F(t) I_2(t) + \left( tF(t) - G(t) \right)I_3(t)\right]
\end{align}
By spherical Crofton's formula~\cite[Section~6.5]{SW08},
\begin{equation}\label{measure of planes}
    \mu_{d,2}\{L \in G_{d,2} \mid L\cap K  \neq \emptyset\} = \frac{|\partial K|}{\omega_{d-1}}.
\end{equation}
Therefore it follows from the definition of $F_{\sigma}$, that 
\begin{equation}\label{part 3}
    I_3(t) = \frac{|\partial K|}{\omega_{d-1}} \left(1 - F_{\sigma}(t)\right).
\end{equation}

To calculate $I_1(t)$ and $I_2(t)$ we will need the following statement:
\begin{lemma}
Let $R: [0, \pi] \rightarrow \mathbb{R}$ be a continuous function. Then
\begin{equation}\label{lemma 1}
    \int\limits_{\alpha(K \cap L) < t} R(\alpha(K \cap L))\mu_{d,2}(\mathrm{d}L) =\frac{|\partial K|}{\omega_{d-1}}\int\limits_0^t R(s) \mathrm{d}F_{\sigma}(s).
\end{equation}
\end{lemma}
\begin{proof}
Consider the function $H(t) = \int\limits_{\alpha(K \cap L) < t} R(\alpha(K \cap L))\mu_{d,2}(\mathrm{d}L)$. We have
\begin{align*}
    &\frac{H(t + \Delta t) - H(t)}{\Delta t} = \frac{1}{\Delta t}\int\limits_{t \leqslant \alpha(K \cap L) < t + \Delta t}R(\alpha(K \cap L))\mu_{d,2}(\mathrm{d}L) \\
    &\quad= R(\theta)\mu_{d,2}\{L \in G_{d,2} \mid L\cap K  \neq \emptyset\}\frac{(F_{\sigma}(t + \Delta t) - F_{\sigma}(t))}{\Delta t}
\end{align*}
for some $\theta \in [t, t + \Delta t]$.

Letting $\Delta t$ to $0$ and using continuity of $R$ and almost everywhere differentiability of $F_{\sigma}$ along with~\eqref{measure of planes} we obtain
\begin{align*}
    \mathrm{d}H(t) = \frac{|\partial K|}{\omega_{d-1}}R(t)\mathrm{d}F_{\sigma}(t),
\end{align*}
and since $H(0) = 0$,
\begin{align*}
    H(t) = \frac{|\partial K|}{\omega_{d-1}}\int\limits_0^t R(s)
\mathrm{d}F_{\sigma}(s).
\end{align*}
\end{proof}

Substituting $x = 0$ in \eqref{whole} gives

\begin{align}\label{G average}
    \int\limits_{L \cap K \neq \emptyset}\left(G(\alpha(K \cap L)) \right)\mu_{d,2}(\mathrm{d}L) &= \frac{|K|^2}{2b_{d,2}} + F(0)\int\limits_{L \cap K \neq \emptyset}\alpha(K \cap L)\mu_{d,2}(\mathrm{d}L)  \\ \notag
    &+ G(0)\frac{|\partial K|}{\omega_{d-1}}. 
\end{align}

Again, by spherical Crofton's formula~\cite[Section~6.5]{SW08},
\begin{equation}\label{angle average}
    \int\limits_{L \cap K \neq \emptyset}\alpha(K \cap L)\mu_{d,2}(\mathrm{d}L)=\frac{2\pi}{\omega_d}|K|.
\end{equation}
Therefore applying \eqref{lemma 1}, \eqref{G average} and \eqref{angle average} leads to
\begin{align}\notag
    I_1(t) &= \int\limits_{L \cap K \neq \emptyset}\left(G(\alpha(K \cap L))\right)\mu_{d,2}(\mathrm{d}L) - \frac{|\partial K|}{\omega_{d-1}}\int\limits_0^t G(s)\mathrm{d}F_{\sigma}(s) \\\notag
    & = \frac{|K|^2}{2b_{d,2}} + F(0)\frac{2\pi}{\omega_d}|K| + G(0)\frac{|\partial K|}{\omega_{d-1}} \\ \label{part 1}
    &- \frac{|\partial K|}{\omega_{d-1}}\left[G(t)F_{\sigma}(t) - \int\limits_0^tF(s)F_{\sigma}(s)\dd s \right]
\end{align}

and
\begin{align}\notag
    I_2(t) &= \int\limits_{L \cap K \neq \emptyset}\alpha(K \cap L)\mu_{d,2}(dL) - \frac{|\partial K|}{\omega_{d-1}}\int\limits_0^t s\mathrm{d}F_{\sigma}(s) \\ \notag
    &= \frac{2\pi}{\omega_d}|K| - \frac{|\partial K|}{\omega_{d-1}}\left(tF_{\sigma}(t) - \int\limits_0^t F_{\sigma}(s)\dd s\right)  \\ \label{part 2}
    &= \frac{2\pi}{\omega_d}|K| - \frac{|\partial K|}{\omega_{d-1}}\left(tF_{\sigma}(t) - \int\limits_0^t F_{\sigma}(s)\dd s\right).
\end{align}

Substituting \eqref{part 1}, \eqref{part 2} and \eqref{part 3} in \eqref{whole} we get:
\begin{align*}
    1 &- F_{\Delta}(t) = \frac{2b_{d,2}}{|K|^2}\Bigg[\frac{|K|^2}{2b_{d,2}} + F(0)\frac{2\pi}{\omega_d}|K| + G(0)\frac{|\partial K|}{\omega_{d-1}}\\
    &+ \frac{|\partial K|}{\omega_{d-1}}\int\limits_0^t F(s)F_{\sigma}(s)\dd s - \frac{|\partial K|}{\omega_{d-1}}F(t)\int\limits_0^tF_{\sigma}(s)\dd s - F(t) \frac{2\pi}{\omega_d}|K| \\
    & + \left(tF(t) - G(t)\right)\frac{|\partial K|}{\omega_{d-1}}\Bigg].
\end{align*}

Therefore, 
\begin{align*}
   F_{\Delta}(t) &= \frac{2b_{d,2}}{|K|^2}\Bigg[\left(F(t) - F(0)\right) \frac{2\pi}{\omega_d}|K|  +  \frac{|\partial K|}{\omega_{d-1}} \Bigg(G(t) - tF(t) - G(0) \\
   &  + \int\limits_0^t\left(F(t) - F(s)\right)F_{\sigma}(s)\dd s\Bigg)\Bigg].
\end{align*}

Differentiating the last equation, we arrive at
\begin{align*}
    f_{\Delta}(t) &= \frac{2b_{d,2}}{|K|^2} \Bigg[F'(t)\frac{2\pi}{\omega_d}|K| \\
    &+  \frac{|\partial K|}{\omega_{d-1}}\left(G'(t) - F(t) - tF'(t) + F'(t)\int\limits_0^t F_{\sigma}(s) \dd s\right) \Bigg] \\
    &= \frac{2b_{d,2}}{|K|^2} \Bigg[(\sin t)      ^{d-2}\frac{2\pi}{\omega_d}|K| -  \frac{|\partial K|}{\omega_{d-1}} (\sin t)^{d-2}\int\limits_0^t \left(1 -F_{\sigma}(s)\right) \dd s \Bigg].
\end{align*}
To conclude the proof it remains to note that 
\begin{align*}
    b_{d,2}=\frac{\omega_d\omega_{d-1}}{4\pi}.
\end{align*}

\section{Acknowledgments}
The work was supported by the Foundation for the Advancement of Theoretical Physics and Mathematics ``BASIS''.

The work of TM was supported by Ministry of Science and Higher Education of the Russian Federation, agreement  075-15-2019-1619.

\bibliographystyle{plain}
\bibliography{bib}

\end{document}

%% file: def.tex
\usepackage[utf8]{inputenc}
\usepackage{graphicx}
\usepackage{amsmath}
\usepackage{amsthm}
\usepackage{amssymb,bbm}%
\usepackage[numbers, square]{natbib}

\theoremstyle{plain}
\newtheorem{theorem}{Theorem}[section]
\newtheorem{lemma}[theorem]{Lemma}
\newtheorem{corollary}[theorem]{Corollary}

\theoremstyle{definition}

\theoremstyle{remark}

\makeindex
\makeglossary

\newcommand{\Sd}{\mathbb{S}^{d-1}}

\newcommand{\E}{\mathbb E\,}
\newcommand{\R}{\mathbb{R}}

\newcommand{\B}{\mathbb{B}}

\newcommand{\conv}{{\mathrm{conv}}}

\newcommand{\bx}{\mathbf{x}}

\newcommand{\dd}{{\,\rm d}}